\documentclass[reqno,b5paper]{amsart}
\usepackage{amsmath}
\usepackage{amssymb}
\usepackage{amsthm}
\usepackage{enumerate}
\usepackage[mathscr]{eucal}
\theoremstyle{plain}
\newtheorem{thm}{Theorem}[section]

\newtheorem{prop}[thm]{Proposition}
\theoremstyle{definition}
\newtheorem{defn}{Definition}[section]

\begin{document}

\setcounter {page}{1}
\title{ Convergence of partial maps via bornology through ideal and its characterization}

\author[P. Malik, Ar. Ghosh]{ Prasanta Malik*, Argha Ghosh*\ }
\newcommand{\acr}{\newline\indent}
\maketitle
\address{{*\,} Department of Mathematics, Burdwan University, Golapbag-713104, Burdwan, W. B, India. Email: pmalik@math.buruniv.ac.in, buagbu@yahoo.co.in\\}

\maketitle
\begin{abstract}In this paper we consider the idea of $I$ - convergence of nets of partial function from a metric space $(X,d)$ to a metric space $(Y,\mu)$ and derive several basic characterization. This idea extends the concept of convergence of nets of partial function introduced by  G. Beer et.al [1].
\end{abstract}
\author{}
\maketitle


\section{\textbf{Introduction}}

We mean by a partial map or a partial function from a metric space $(X,d)$ to a metric space $(Y,\mu)$, a pair $(D,u)$ where $D$ is a nonempty closed subset of $X$ and $u:~ D\rightarrow Y$ is a function. Let us denote the set of all such partial maps by $\mathcal P[X,Y]$ and by $\mathcal C[X,Y]$ we mean those partial maps which are continuous on their respective domains. The notion of convergence of partial maps was first introduced by G. Beer et.al.[1]. The notion of $I$-convergence of nets was first introduced by B. K. Lahiri and P. Das [8]. In this paper we will use this two notion and introduced a new type of convergence on partial maps which will produced a new research area.

\section{\textbf{Preliminaries  }}
In this section we give some basic definitions and discuss some ideas which will helpful to understand this paper in the next section.
\begin{defn}$[]$
 If $X$ is a non-void set, then a family $I\subset 2^X$ is
called an ideal if\\
(i) $\phi\in I$ and\\
(ii) $A~;~B \in I$ implies $A \cup B \in I$ and\\
(iii) $A \in I$ ; $B\subset A $ implies $B \in I$.\\\\
The ideal $I$ is called non-trivial if $I\neq\left\{\phi\right\}$ and $X\notin I$.
\end{defn}
\begin{defn}$[]$
A non empty family $\mathbb F$ of subsets of a non-void set $X$ is a filter if\\
(i) $\phi\notin \mathbb F$ and\\
(ii) $A~;~B\in \mathbb F$ implies $A \cap B \in\mathbb F$ and\\
(iii) $A \in\mathbb F$ ; $A \subset B$ implies $B \in\mathbb F$.\\\\
Clearly $I \subset 2^X$ is a non-trivial ideal of $X$ if and only if $\mathbb F = \mathbb F(I) =\left\{A\subset X~:~ X\setminus A\in I\right\}$ is a filter on $X$, called the filter associated with $I$.
\end{defn}
The following two definitions are well known to all but for sake of completeness we give it below
\begin{defn}
 Let $\Gamma$ be a non-void set and let $\geq$ be a binary relation on $\Gamma$
such that $\geq$ is reflexive, transitive and for any two elements $m~;~ n \in\Gamma $, there
is an element $p \in\Gamma$ such that $p \geq m$ and $p \geq n$. The pair $(\Gamma;\geq)$ is called a
directed set.
\end{defn}
\begin{defn}
 Let $(\Gamma;\geq)$ be a directed set and let $X$ be a non-void set. A
mapping $\gamma :\Gamma\rightarrow X$ is called a net in $X$ denoted by $\left\{\gamma_n:n\in\Gamma\right\}$ or simply by
$\left\{\gamma_n\right\}$ when the set $\Gamma$ is clear.
\end{defn}
Throughout the paper $X = (X, d)$ and $Y = (Y,\mu)$ will denote metric spaces. We write $CL(X)$ for the
collection of the closed nonempty subsets of $X$, $K(X)$ is the collection of the compact nonempty subsets of
$X$. And by $\mathbb N$ we denote the set of all natural numbers and $I$ will denote a non-trivial ideal of a directed set $\Gamma$.

For $n \in\Gamma$ let $M_n = \left\{k\in\Gamma:k\geq n\right\}$. Then the collection 
\begin{center}
$\mathbb F_0 =\left\{A\subset\Gamma:A\supset M_n~for~some~n\right\}$
\end{center}
 forms a filter in $\gamma$. Let $I_0 =\left\{A\subset\Gamma:\Gamma\setminus A\in\mathbb F_0\right\}$. Then $I_0$ is also a non-trivial ideal in $\Gamma$.
\begin{defn}
 A non-trivial ideal $I$ of $\Gamma$ will be called $D-admissible$ if $M_n \in\mathbb F(I)$ for all $n \in \Gamma$.
\end{defn}
We now discuss the notion of bornology (for more details see [])\\
If $x_0 \in X$ and $\epsilon > 0$, $B(x_0, \epsilon)$ is the open $\epsilon$-ball with center $x_0$ and radius $\epsilon$. If $A$ is a nonempty subset of $X$,
we write $d(x_0,A)$ for the distance from $x_0$ to $A$. We denote by $A^\epsilon$ the $\epsilon$-enlargement of the set $A:$
\begin{center}
$A^\epsilon=\left\{x:d(x,A)<\epsilon\right\}=\underset{x\in A}{\bigcup}B(x,\epsilon)$.
\end{center}
\begin{defn}
A \textsl{bornology} $\mathcal B$ on a metric space $(X, d)$ is a family of subsets of $X$, covering $X$, closed under taking
finite unions, and hereditary, i.e., closed under taking nonempty subsets.
\end{defn}
The smallest bornology on $X$ is the family of the finite subsets of $X$, $\mathcal F$ , and the largest is the family of all non
empty subsets of $X$, $P_0(X)$. Other important bornologies are: the family $B_d$ of the nonempty $d$-bounded
subsets, the family $B_{tb}$ of the nonempty $d$-totally bounded subsets and the family $\mathcal K$ of nonempty subsets
of $X$ whose closure sets are compact.\\
We now give some basic definition related to bornological convergence as defined in([],[]).
\begin{defn}
Let $(X, d)$ be a metric space and $\mathcal B$ be a bornology on $(X,d)$. A net $\left\langle D_\gamma\right\rangle_{\gamma\in \Gamma}$ in $\mathcal P_0(X)$ is called $\mathcal B^-$-convergent (lower bornological convergent) to $D\in \mathcal P_0(X)$ if for every $B\in\mathcal B$ and $\epsilon>0$, the following inclusion holds eventually:
\begin{center}
 $D\cap B\subset D_\gamma^\epsilon.$ 
\end{center}
In this case we shall write $D\in\mathcal B^--lim ~D_\gamma$ when this holds. Similarly the net is called $\mathcal B^+$-convergent (upper bornological convergent) to $D\in\mathcal P_0(x)$ if for every $B\in\mathcal B$ and $\epsilon>0$,
\begin{center}
 $D_\gamma\cap B\subset D^\epsilon$.
\end{center}
  In this case we shall write $D\in\mathcal B^+-lim ~D_\gamma$ when this occurs.
\end{defn}
 
Naturally two-sided bornological convergence occurs when both upper and lower convergences occur, and we then write $D\in\mathcal B-lim ~D_\gamma$.
\begin{defn}
Let $(X,d)$, $(Y,\mu)$ be metric spaces, and $\mathcal B$ be a bornology on $X$. Let $\Gamma$ be a directed set and let $\left\langle \left\langle D_\gamma,u_\gamma\right\rangle\right\rangle_{\gamma\in\Gamma}$ be a net in $\mathcal P[X,Y].$ We say that the net is $\mathcal P(\mathcal B)$-convergent to $(D,u)$, we write $(D,u)\in\mathcal P(\mathcal B)-lim(D_\gamma,u_\gamma)$, if for every $B\in \mathcal B$ and $\epsilon>0$, the following two conditions hold for all indices $\gamma\geq\gamma_0$
\begin{center}
(i) for each nonempty subset $B_1$ of $ B,~u(D\cap B_1)\subset[U_\gamma(D_\gamma\cap B_1^\epsilon)]^\epsilon$
\end{center}
\begin{center}
(ii) for each nonempty subset $B_1$ of $ B,~u_\gamma(D_\gamma\cap B_1)\subset[U(D\cap B_1^\epsilon)]^\epsilon$.
\end{center}
\end{defn}

The most tangible and visual description of $\mathcal P(\mathcal B)$-convergence is the following: for each $B\in\mathcal B$ and $\epsilon>0$, eventually both $Gr(u_\gamma)\cap(B\times Y)\subset Gr(u)^\epsilon$ and $Gr(u)\cap(B\times Y)\subset Gr(u_\gamma)^\epsilon$. In this formulation, the enlargement is taken with respect to any metric compatible with the product uniformity. For definiteness, we choose the\textit{ box metric} defined by 
\begin{center}
$(d\times\mu)((x_1,y_1),(x_2,y_2)):=max\left\{d(x_1,x_2),\mu(y_1,y_2)\right\}$.
\end{center}
\begin{defn}
Let $(X,d)$, $(Y,\mu)$ be metric spaces, and $\mathcal B$ be a bornology on $X$. Let $\Gamma$ be a directed set and $\left\langle \left\langle D_\gamma,u_\gamma\right\rangle\right\rangle_{\gamma\in\Gamma}$ be a net in $\mathcal P[X,Y].$ We say that the net is $\mathcal P^-(\mathcal B)$-convergent to $(D,u)$, we write $(D,u)\in\mathcal P^-(\mathcal B)-lim(D_\gamma,u_\gamma)$, if for every $B\in \mathcal B$ and $\epsilon>0$, 
\begin{center}
for all $B_1(\subset B),~u(D\cap B_1)\subset[U_\gamma(D_\gamma\cap B_1^\epsilon)]^\epsilon$
\end{center}
holds. Similarly we can say the net is $\mathcal P^+(\mathcal B)$-convergent to $(D,u)$, we write $(D,u)\in\mathcal P^+(\mathcal B)-lim(D_\gamma,u_\gamma)$, if for every $B\in \mathcal B$ and $\epsilon>0$,
\begin{center}
for all $B_1(\subset B),~u_\gamma(D_\gamma\cap B_1)\subset[U(D\cap B_1^\epsilon)]^\epsilon$
\end{center}
holds.
\end{defn}

\section{\textbf{Main Results }}

$~~$ Let $(X, d)$ and $(Y,\mu)$ be metric spaces. In this section we investigate the notion of convergence of partial maps by ideals of directed sets. So we first give some definition.
\begin{defn}
Let $(X,d)$, $(Y,\mu)$ be metric spaces, and $\mathcal B$ be a bornology on $X$. Let $\Gamma$ be a directed set and $I$ be a nontrivial ideal of $\Gamma$ and $\left\langle \left\langle D_\gamma,u_\gamma\right\rangle\right\rangle_{\gamma\in\Gamma}$ be a net in $\mathcal P[X,Y].$ We say that the net is $\mathcal P_I(\mathcal B)$-convergent to $(D,u)$, we write $(D,u)\in\mathcal P_I(\mathcal B)-lim(D_\gamma,u_\gamma)$, if for every $B\in \mathcal B$ and $\epsilon>0$, the following two criteria 
\begin{center}
$(i)~~\left\{\gamma:\forall B_1(\subset B),~u(D\cap B_1)\subset[U_\gamma(D_\gamma\cap B_1^\epsilon)]^\epsilon\right\}\in\mathbb F(I).$
\end{center}
\begin{center}
$(ii)~~\left\{\gamma:\forall B_1(\subset B),~u_\gamma(D_\gamma\cap B_1)\subset[U(D\cap B_1^\epsilon)]^\epsilon\right\}\in\mathbb F(I)$
\end{center}
hold.
\end{defn}
\begin{defn}
Let $(X, d)$ be a metric space and $\mathcal B$ be a bornology on $(X,d)$. A net $\left\langle D_\gamma\right\rangle_{\gamma\in \Gamma}$ in $\mathcal P_0(X)$ is called $\mathcal B_I^-$-convergent (lower bornological $I$-convergent) to $D\in \mathcal P_0(X)$ if for every $B\in\mathcal B$ and $\epsilon>0$,
\begin{center}
 $\left\{\gamma:~D\cap B\subset D_\gamma^\epsilon\right\}\in\mathbb F(I).$ In this case we shall write $D\in\mathcal B_I^--lim ~D_\gamma$
\end{center}
when this holds. Similarly the net is called $\mathcal B_I^+$-convergent (upper bornological $I$-convergent) to $D\in\mathcal P_0(x)$ if for every $B\in\mathcal B$ and $\epsilon>0$,
\begin{center}
 $\left\{\gamma:~D_\gamma\cap B\subset D^\epsilon\right\}\in\mathbb F(I).$ In this case we shall write $D\in\mathcal B_I^+-lim ~D_\gamma$
\end{center}
when this occurs.
\end{defn}
\begin{defn}
Let $(X,d)$, $(Y,\mu)$ be metric spaces, and $\mathcal B$ be a bornology on $X$. Let $\Gamma$ be a directed set and $I$ be a nontrivial ideal of $\Gamma$ and $\left\langle \left\langle D_\gamma,u_\gamma\right\rangle\right\rangle_{\gamma\in\Gamma}$ be a net in $\mathcal P[X,Y].$ We say that the net is $\mathcal P_I^-(\mathcal B)$-convergent to $(D,u)$, we write $(D,u)\in\mathcal P_I^-(\mathcal B)-lim(D_\gamma,u_\gamma)$, if for every $B\in \mathcal B$ and $\epsilon>0$, 
\begin{center}
$\left\{\gamma:\forall B_1(\subset B),~u(D\cap B_1)\subset[U_\gamma(D_\gamma\cap B_1^\epsilon)]^\epsilon\right\}\in\mathbb F(I).$
\end{center}
holds. Similarly we can say the net is $\mathcal P_I^+(\mathcal B)$-convergent to $(D,u)$, we write $(D,u)\in\mathcal P_I^+(\mathcal B)-lim(D_\gamma,u_\gamma)$, if for every $B\in \mathcal B$ and $\epsilon>0$,
\begin{center}
$\left\{\gamma:\forall B_1(\subset B),~u_\gamma(D_\gamma\cap B_1)\subset[U(D\cap B_1^\epsilon)]^\epsilon\right\}\in\mathbb F(I)$
\end{center}
holds.
\end{defn}
\begin{prop}
Let $\left\langle \left\langle D_\gamma,u_\gamma\right\rangle\right\rangle_{\gamma\in\Gamma}$ be a net in $\mathcal P[X,Y]$, $\mathcal B$ be a bornology on the metric space $(X,d)$ and $I$ is an ideal of $\Gamma$.
\begin{center}
$(i)$ If $(D,u)\in \mathcal P_I^-(\mathcal B)-lim~(D_\gamma,u_\gamma)$, then $\forall B\in \mathcal B$ and $\forall\epsilon>0$,  $\left\{\gamma:~D\cap B\subset D_\gamma^\epsilon\right\}\in\mathbb F(I).$
\end{center}
\begin{center}
$(ii)$ If $(D,u)\in \mathcal P_I^+(\mathcal B)-lim~(D_\gamma,u_\gamma)$, then $\forall B\in \mathcal B$ and $\forall\epsilon>0$,  $\left\{\gamma:~D_\gamma\cap B\subset D^\epsilon\right\}\in\mathbb F(I).$
\end{center}
\end{prop}
\begin{proof}
We only prove statement (i), one can prove statement (ii) similarly. Let $B\in\mathcal B$ and $\epsilon>0$ be given. By assumption we have 
\begin{center}
$A=\left\{\gamma:~\forall B_1(\subset B), u(D\cap B_1)\subset[u_\gamma(D_\gamma\cap B_1^\epsilon)]^\epsilon\right\}\in\mathbb F(I).$
\end{center}
 Since $I$ is nontrivial, we can choose $\gamma_1\in A$. Let $x\in D\cap B$. Now  with $B_1=\left\{x\right\}$ we get 
 \begin{center}
 $u(x)\in[u_{\gamma_1}(D_{\gamma_1}\cap\left\{x\right\}^\epsilon)]^\epsilon$.
 \end{center}
This means that for some $v\in D_{\gamma_1}\cap B_d(x,\epsilon)$ we have $\mu(u(x),u_{\gamma_1}(v))<\epsilon$. More particularly, $x\in B_d(v,\epsilon)\subset D_{\gamma_1}^\epsilon.$ Since $x\in D\cap B$ is arbitrary thus 
\begin{center}
$A\subset \left\{D\cap B\subset D_\gamma^\epsilon\right\}$. As $A\in\mathbb F(I)$ thus the later set. This completes the proof.
\end{center}

\end{proof}
\begin{prop}
 The condition  for every $B\in \mathcal B$ and $\epsilon>0$,
\begin{center}
$\left\{\gamma:\forall B_1(\subset B),~u_\gamma(D_\gamma\cap B_1)\subset[u(D\cap B_1^\epsilon)]^\epsilon\right\}\in\mathbb F(I)$
\end{center}
holds if and only if  for every $B\in \mathcal B$ and $\epsilon>0$,
\begin{center}
$\left\{\gamma:\underset{z\in D_\gamma\cap B}{\sup}~~\underset{x\in B_d(z,\epsilon)}{\inf}\mu(u(x),u_\gamma(z))<\epsilon\right\}\in\mathbb F(I)$
\end{center}
holds.
\end{prop}
\begin{proof}
First we prove necessary part of the proposition. Let $B\in\mathcal B$ and $\epsilon>0$ be given. Then by assumption we have 
\begin{center}
$A=\left\{\gamma:\forall B_1(\subset B),~u_\gamma(D_\gamma\cap B_1)\subset[u(D\cap B_1^{\frac{\epsilon}{2}})]^{\frac{\epsilon}{2}}\right\}\in\mathbb F(I).$
\end{center}
Let us choose $z\in D_\gamma\cap B$ with $B_1=\{z\}$, $\gamma\in A$ we have $u_\gamma(z)\in [u(D\cap B_1^{\frac{\epsilon}{2}})]^{\frac{\epsilon}{2}}$. This means that for some $x\in D\cap B_d(z,\frac{\epsilon}{2})$ implies $\mu(u_\gamma(z),u(x))<\frac{\epsilon}{2}$. Then clearly $\underset{x\in D\cap B_d(z,\frac{\epsilon}{2})}{\inf}\mu(u_\gamma(z),u(x))<\frac{\epsilon}{2}$. Also 
\begin{center}
$\underset{z\in D_\gamma\cap B}{\sup}\underset{x\in D\cap B_d(z,\frac{\epsilon}{2})}{\inf}\mu(u_\gamma(z),u(x))\leq\frac{\epsilon}{2}<\epsilon$.
\end{center}
  But again
	\begin{center}
	 $\underset{z\in D_\gamma\cap B}{\sup}\underset{x\in D\cap B_d(z,\epsilon)}{\inf}\mu(u_\gamma(z),u(x))<\underset{z\in D_\gamma\cap B}{\sup}\underset{x\in D\cap B_d(z,\frac{\epsilon}{2})}{\inf}\mu(u_\gamma(z),u(x))<\epsilon$.
	\end{center}
	 Thus $A\subset \left\{\gamma:\underset{z\in D_\gamma\cap B}{\sup}\underset{x\in D\cap B_d(z,\epsilon)}{\inf}\mu(u_\gamma(z),u(x))<\epsilon\right\}.$  Since $A\in\mathbb F(I)$ thus the later one in $\mathbb F(I).$ Therefore the condition is necessary.
	
	Now we prove the sufficient part of the proposition. Suppose 
	\begin{center}
	$\{\gamma:\underset{z\in D_\gamma\cap B}{\sup}\underset{x\in D\cap B_d(z,\frac{\epsilon}{2})}{\inf}\mu(u_\gamma(z),u(x))<\epsilon\}\in\mathbb F(I)$.
	\end{center}
	 Clearly if $B_1\subset B$, we have 
	\begin{center}
	$A_1=\{\gamma:\underset{z\in D_\gamma\cap B_1}{\sup}\underset{x\in D\cap B_d(z,\frac{\epsilon}{2})}{\inf}\mu(u_\gamma(z),u(x))<\epsilon\}\in\mathbb F(I)$.
	\end{center}
	Let $\gamma\in A_1$ then we have for all $z\in D_\gamma\cap B_1$ there exists $x\in D\cap \{z\}^\epsilon\subset D\cap B_1^\epsilon$ with $\mu(u_\gamma(z),u(x))<\epsilon$. Thus 
	\begin{center}
	$A_1\subset\left\{\gamma:\forall B_1(\subset B),~u_\gamma(D_\gamma\cap B_1)\subset[u(D\cap B_1^\epsilon)]^\epsilon\right\}$.
	\end{center}
	But $A_1\in\mathbb F(I)$, thus the later set in $\mathbb F(I).$ Hence the condition is sufficient.
\end{proof}
 Similarly one can prove the following result.
\begin{prop}
 The condition  for every $B\in \mathcal B$ and $\epsilon>0$,
\begin{center}
$\left\{\gamma:\forall B_1(\subset B),~u(D\cap B_1)\subset[u_\gamma(D_\gamma\cap B_1^\epsilon)]^\epsilon\right\}\in\mathbb F(I)$
\end{center}
holds if and only if  for every $B\in \mathcal B$ and $\epsilon>0$,
\begin{center}
$\left\{\gamma:\underset{z\in D\cap B}{\sup}~~\underset{x\in B_d(z,\epsilon)\cap D_\gamma}{\inf}\mu(u(z),u_\gamma(x))<\epsilon\right\}\in\mathbb F(I)$
\end{center}
holds.
\end{prop}

\begin{thm}
Let $\left\langle \left\langle D_\gamma,u_\gamma\right\rangle\right\rangle_{\gamma\in\Gamma}$ be a net of partial functions from the metric space $(X,d)$ to the metric space $(Y,\mu)$ and let $I$ be an ideal of $\Gamma$ and $\mathcal B$ be a bornology on $X$. Then for $(D,u)\in \mathcal P[X,Y]$ 
\begin{center}
$(1)~~ Gr(u)\in\mathcal (B_I^*)^--lim~Gr(u_\gamma)$ if and only if $(D,u)\in\mathcal P_I^-(\mathcal B)-lim(D_\gamma,u_\gamma)$;
\end{center}
\begin{center}
$(2)~~ Gr(u)\in\mathcal (B_I^*)^+-lim~Gr(u_\gamma)$ if and only if $(D,u)\in\mathcal P_I^+(\mathcal B)-lim(D_\gamma,u_\gamma)$.
\end{center}
\end{thm}
\begin{proof}
We just verify statement (1). Suppose $(D,u)\in\mathcal P_I^-(\mathcal B)-lim(D_\gamma,u_\gamma)$. To verify bornology convergence of graphs, it is suffices to work with the basic sets in $\mathcal B^*$. Let $B\times Y$ be such a basic set where $B\in\mathcal B$. Let $\epsilon>0$ be given, we have by assumption 
\begin{center}
$A=\{\gamma:D\cap B\subset D_\gamma^\epsilon\}\in \mathbb F(I)$ and
\end{center}
\begin{center}
$B=\left\{\gamma:\underset{z\in D\cap B}{\sup}~~\underset{x\in B_d(z,\epsilon)\cap D_\gamma}{\inf}\mu(u(z),u_\gamma(x))<\epsilon\right\}\in\mathbb F(I)$.
\end{center}
But $A, ~B\in\mathbb F(I)\Rightarrow A\cap B\in\mathbb F(I).$ Choose $\gamma\in A\cap B$ then both 
\begin{center}
$(i)~~ D\cap B\subset D_\gamma^\epsilon$
\end{center}
\begin{center}
$(ii)~~\underset{z\in D\cap B}{\sup}~~\underset{x\in B_d(z,\epsilon)\cap D_\gamma}{\inf}\mu(u(z),u_\gamma(x))<\epsilon $
\end{center}
holds. Also for $\gamma\in A\cap B$ and $(z,u(z))\in (B\times Y)\cap Gr(u)$ so that $z\in D$. By $(i)$ $B_d(z,\epsilon)\cap D_\gamma\neq \phi$ and by $(ii)$ for some $x\in B_d(z,\epsilon)\cap D_\gamma$, we have $\mu(u_\gamma,u(z))<\epsilon.$ So we have $(x,u_\gamma(x))\in Gr(u_\gamma)$ and $(d\times \mu)((z,u(z))(x,u_\gamma(x)))<\epsilon$ and this yields $Gr(u)\cap (B\times Y)\subset Gr^\epsilon(u_\gamma)$. Therefore 
\begin{center}
$A\cap B\subset \{\gamma:~ Gr(u)\cap(B\times Y)\subset Gr^\epsilon(u_\gamma)\}.$
\end{center}
Since $A\cap B \in\mathbb F(I)$, so the later set.

 Conversely, we consider the lower bornological $I$-convergence of graphs. Let $B\in\mathcal B$ and $\epsilon>0$ be given. Choosing $0<\eta<\epsilon$, we have
 \begin{center}
	$A_1=\{\gamma:~ Gr(u)\cap(B\times Y)\subset Gr^\eta(u_\gamma)\}\in \mathbb F(I).$ 
 \end{center}
Let $\gamma\in A_1$ and let $z\in D\cap B$ be arbitrary. Clearly $(z,u(z))\in (B\times Y)\cap Gr(u)$, so there exists $(y_0,u_\gamma(y_0))\in Gr(u_\gamma)$ such that 
\begin{center}
$(d\times\mu)((z,u(z))(y_0,u_\gamma(y_0)))<\eta$.
\end{center}
Thus $d(z,y_0)<\eta<\epsilon$. By the same argument we can say $\mu(u(z),u_\gamma(y_0))<\eta$, so that $\underset{x\in B_d(z,\epsilon)\cap D_\gamma}{\inf}\mu(u(z),u_\gamma(x))<\eta$ and by taking the supremum over $z\in D\cap B$ we have 
\begin{center}
$\underset{z\in D\cap B}{\sup}\underset{x\in B_d(z,\epsilon)\cap D_\gamma}{\inf}\mu(u(z),u_\gamma(x))\leq\eta<\epsilon.$
\end{center}
Thus 
\begin{center}
$A_1\subset\{\gamma:\underset{z\in D\cap B}{\sup}\underset{x\in B_d(z,\epsilon)\cap D_\gamma}{\inf}\mu(u(z),u_\gamma(x))<\epsilon\}.$
\end{center}
Since $A_1\in\mathbb F(I)$ hence we have the required results.

\end{proof}

\begin{defn}
Let $\left\langle \left\langle D_\gamma,u_\gamma\right\rangle\right\rangle_{\gamma\in \Gamma}$ be a net of partial maps, $(D,u)\in\mathcal P[X,Y]$ and $I$ be an ideal of $\Gamma$. We say that $\left\langle \left\langle D_\gamma,u_\gamma\right\rangle\right\rangle_{\gamma\in \Gamma}$ is $I$-converges pointwise to $(D,u)$ if whenever $x\in D_\gamma$ for a cofinal subset $\Gamma_0\subset \Gamma$, then $x\in D$ and $u(x)=I_{\Gamma_0}-lim~u_\gamma(x)$.
\end{defn}
\begin{prop}
Let $\left\langle \left\langle D_\gamma,u_\gamma\right\rangle\right\rangle_{\gamma\in\Gamma}$ be a net in $\mathcal P[X,Y]$. If the net is $\mathcal P_I^+(\mathcal B)$-convergent to $(D,u)\in\mathcal C[X,Y]$, then it is pointwise $I$-convergent to $(D,u)$.
\end{prop}
\begin{proof}
Let $\Gamma_0$ be a cofinal set of $\Gamma$ and let $\gamma\in\Gamma_0$ with $x\in D_\gamma$. By proposition 3.1, we have $x\in D$ because $D$ is closed. Again by continuity of $u$, we can choose $\delta<\frac{\epsilon}{2}$ such that if $d(z,x)<\delta$ then $\mu(u(z),u(x))<\frac{\epsilon}{2}.$ Also by $\mathcal (B_I^*)^+$-convergence of graphs 
\begin{center}
$B=\{\gamma\in\Gamma_0:(\{x\}\times Y)\cap Gr(u_\gamma)\subset Gr^\delta(u)\}\in \mathbb F(I_{\Gamma_0})$.
\end{center}
Choose $\gamma\in B$, then there exists $z_\gamma\in D$ with 
\begin{center}
$(d\times\mu)((x,u_\gamma(x))(z_\gamma,u(z_\gamma)))<\delta$.
\end{center}
So we have $\mu(u_\gamma(x),u(z_\gamma))<\delta$ and $d(x,z_\gamma)<\delta$. Again by triangle inequality 
\begin{center}
$\mu(u_\gamma(x),u(x))\leq \mu(u_\gamma(x),u(z_\gamma))+\mu(u(x),u(z_\gamma))<\frac{\epsilon}{2}+\frac{\epsilon}{2}=\epsilon$
\end{center}
Hence $B\subset\left\{\gamma:\mu(u_\gamma(x),u(x))<\epsilon\right\}$. Since $B\in \mathbb F(I_{\Gamma_0})$ thus the later set. This completes the proof.
\end{proof}
\begin{thm}
Let $\mathcal B$ be a bornology on $(X,d)$ and let $(D,u)$ be strongly uniformly continuous relative to $\mathcal B$ with values in $(Y,\mu)$. Then a net $\left\langle \left\langle D_\gamma,u_\gamma\right\rangle\right\rangle_{\gamma\in\Gamma}$ in $\mathcal P[X,Y]$ is $\mathcal P_I^+(\mathcal B)$-convergent to $(D,u)$ is equivalent to the condition $\forall B\in\mathcal B$ and $\epsilon>0$, there exists $\zeta>0$ such that 
\begin{center}
$\left\{\gamma:\underset{z\in D_\gamma\cap B}{\sup}~~\underset{x\in B_d(z,\zeta)\cap D}{\sup}~\mu(u(x),u_\gamma(z))<\epsilon\right\}\in\mathbb F(I).$ 
\end{center}

\end{thm}
\begin{proof}
Sufficient part of the statement follows from proposition 3.5. We only need to show that upper bornological $I$-convergence implies the sup-sup condition above.

 Let $B\in\mathcal B$ and $\epsilon>0$ be given. Let $\eta>0$ be such that $2\eta<\epsilon$. By strong uniform continuity of $u$ relative to $\mathcal B$, there exists $\delta$, $0<\delta<2\eta$ such that $x,y\in D\cap B^\delta$ and $d(x,y)<\delta$ implies $\mu(u(x),u(y))<\eta$. Again by assumption 
 \begin{center}
 $A=\left\{\gamma:~D_\gamma\cap B\subset D^{\frac{\delta}{2}}\right\}\in\mathbb F(I)$
 \end{center}
and
\begin{center}
 $B=\left\{\gamma:~\underset{z\in D_\gamma\cap B}{\sup}~~\underset{x\in B_d(z,\frac{\delta}{2})}{\inf}~\mu(u(x),u_\gamma(z))<\frac{\delta}{2}\right\}\in\mathbb F(I).$
\end{center}
Thus $A\cap B\in \mathbb F(I)$. Choose $\gamma\in A\cap B$. Then for every $z\in B\cap D_\gamma$, there exists $x_z\in B_d(z,\frac{\delta}{2})\cap D$ such that $\mu(u(x_z),u_\gamma(z))<\frac{\delta}{2}<\eta$. But since $B_d(z,\frac{\delta}{2})\subset B^{\delta}$, for every $x\in B_d(z,\frac{\delta}{2})\cap D\subset B^\delta\cap D$ we have by strong uniform continuity $\mu(u(x),u(x_z))<\eta$ because $d(x,x_z)<\delta$ . Thus for every $x\in B_d(z,\frac{\delta}{2})\cap D$ 
\begin{center}
$\mu(u(x),u_\gamma(z))\leq\mu(u(x),u(x_z))+\mu(u(x_z),u_\gamma(z))<\eta+\eta=2\eta$
\end{center}
and hence
\begin{center}
 $\underset{z\in D_\gamma\cap B}{\sup}~~\underset{x\in B_d(z,\frac{\delta}{2})\cap D}{\sup}~\mu(u(x),u_\gamma(z))\leq 2\eta<\delta$.
\end{center}
Choose $\zeta=\frac{\delta}{2}$ then we have 
\begin{center}
$A\cap B\subset \left\{\gamma:\underset{z\in D_\gamma\cap B}{\sup}~~\underset{x\in B_d(z,\zeta)\cap D}{\sup}~\mu(u(x),u_\gamma(z))<\epsilon\right\}.$ 
\end{center}
Since $A\cap B\in\mathbb F(I)$ thus the later set. This yields to the prove.
\end{proof}
\begin{thm}
Let $\mathcal B$ be a bornology on $(X,d)$ that is stable under small enlargement and let $(D,u)$ be uniformly continuous relative to $\mathcal B$ with values in $(Y,\mu)$. Then a net $\left\langle \left\langle D_\gamma,u_\gamma\right\rangle\right\rangle_{\gamma\in\Gamma}$ in $\mathcal P[X,Y]$ is $\mathcal P_I(\mathcal B)$-convergent to $(D,u)$ if and only if both the following two condition hold:\\\\
 $(1)~~$for each $B\in\mathcal B$ and $\epsilon>0$, there exists $\zeta>0$ such that 
\begin{center}
$\left\{\gamma:\underset{z\in D_\gamma\cap B}{\sup}~~\underset{x\in B_d(z,\zeta)\cap D}{\sup}~\mu(u(x),u_\gamma(z))<\epsilon\right\}\in\mathbb F(I)$ ;
\end{center}
$(2)~~$ for each $B\in \mathcal B$ and $\epsilon>0$,
\begin{center}
$\left\{\gamma: D\cap B\subset D_\gamma^\epsilon\right\}\in\mathbb F(I)$.
\end{center}

\end{thm}
\begin{proof}
Necessity follows from that $(D,u)$ uniformly is continuous relative to $\mathcal B$ implies strong uniform continuity relative to $\mathcal B$, by proposition 3.1 and theorem 3.2.
Let $B\in\mathcal B$ and $\epsilon>0$ be given, we only need to show that 
\begin{center}
$\left\{\gamma:\underset{z\in D\cap B}{\sup}~~\underset{x\in B_d(z,\epsilon)\cap D_\gamma}{\inf}~\mu(u(z),u_\gamma(x))<\epsilon\right\}\in\mathbb F(I)$.
\end{center}
Choose $\zeta<\epsilon$ so small that $B^\zeta\in\mathcal B$ and satisfies 
\begin{center}
$C=\left\{\gamma:~\underset{z\in D_\gamma\cap B^\zeta}{\sup}~~\underset{x\in B_d(z,\zeta)\cap D}{\sup}~\mu(u(x),u_\gamma(z))<\frac{\epsilon}{2}\right\}\in\mathbb F(I)$
\end{center}
and
\begin{center}
$E=\{\gamma: D\cap B\subset D_\gamma^\zeta\}\in \mathbb F(I)$.
\end{center}
Choose $\gamma\in C\cap E$ and let $z\in D\cap B $, so there exists $x(z,\gamma)\in D_\gamma\cap B^\zeta$ with $d(z,x(z,\gamma))<\zeta$. Again $\gamma\in C$, $x(z,\gamma)\in D_\gamma\cap B^\zeta$ and $z\in B_d(z,\zeta)\cap D$ implies $\mu(u(z),u_\gamma(x(z,\gamma)))<\frac{\epsilon}{2}$. Hence $\underset{x\in B_d(z,\zeta)\cap D_\gamma}{\inf}~\mu(u(z),u_\gamma(x(z,\gamma)))<\frac{\epsilon}{2}$. Also 
\begin{center}
$\underset{z\in D\cap B}{\sup}~~\underset{x\in B_d(z,\epsilon)\cap D_\gamma}{\inf}\mu(u(z),u_\gamma(x))\leq\underset{z\in D\cap B}{\sup}~~\underset{x\in B_d(z,\zeta)\cap D_\gamma}{\inf}\mu(u(z),u_\gamma(x))\leq \frac{\epsilon}{2}<\epsilon$. 
\end{center}
Therefore $C\cap E\subset \left\{\gamma:\underset{z\in D\cap B}{\sup}~~\underset{x\in B_d(z,\epsilon)\cap D_\gamma}{\inf}\mu(u(z),u_\gamma(x))<\epsilon\right\}.$ Now $C\cap E\in \mathbb F(I)$. So the required result.
 
\end{proof}
\noindent\textbf{Acknowledgement:} The work of the second author was supported by University Grants Commission, New Delhi, India, under UGC-JRF in Mathematical Sciences.\\


\end{document}